\documentclass[sn-mathphys,Numbered]{sn-jnl}

\usepackage{graphicx}%
\usepackage{multirow}%
\usepackage{amsmath,amssymb,amsfonts}%
\usepackage{amsthm}%
\usepackage{mathrsfs}%
\usepackage[title]{appendix}%
\usepackage{xcolor}%
\usepackage{textcomp}%
\usepackage{manyfoot}%
\usepackage{booktabs}%
\usepackage{algorithm}%
\usepackage{algorithmicx}%
\usepackage{algpseudocode}%
\usepackage{listings}%

\def\F{\mathbb{F}}

\theoremstyle{thmstyleone}%
\newtheorem{theorem}{Theorem}
\newtheorem{proposition}[theorem]{Proposition}%

\theoremstyle{thmstyletwo}%
\newtheorem{remark}{Remark}%
\newtheorem{lemma}{Lemma}

\newtheorem{conjecture}{Conjecture}
\theoremstyle{thmstylethree}%

\begin{document}

\title[Article Title]{Permutation and local permutation polynomials of maximum degree}

\author*[1]{\fnm{Jaime} \sur{Gutierrez}}\email{jaime.gutierrez@unican.es}

\author*[2]{\fnm{Jorge} \sur{Jiménez Urroz}}\email{jorge.urroz@upm.es}

\affil[1]{\orgdiv{Departamento de Matem\'atica Aplicada y Ciencias de la Computaci\'on}, \orgname{Universidad de Cantabria}, 
\orgaddress{\street{Avda Los Castros}, \city{Santander}, 
 \country{Spain}}}

\affil[2]{\orgdiv{Departamento de Matem\'atica e Informática aplicadas a la ingeniería civil y naval}, \orgname{Universidad Politécnica de Madrid}, \orgaddress{\street{C/ Profesor Aranguren, 3}, \city{Madrid}, \country{Spain}}}

\abstract{Let $\F_q$ be the finite field with $q$ elements and $\F_q[x_1,\ldots, x_n]$ the ring of polynomials in $n$ variables over $\F_q$. 
 In this paper we consider  permutation  polynomials  and local permutation polynomials  over $\F_q[x_1,\ldots, x_n]$,  which  define interesting generalizations of permutations over finite fields. We are able to  construct permutation polynomials in $\F_q[x_1,\ldots, x_n]$ of maximum degree $n(q-1)-1$ and   local permutation polynomials  in  $\F_q[x_1,\ldots, x_n]$ of maximum degree $n(q-2)$ when $q>3$,  extending previous results.}




\maketitle

\section{Introduction}\label{sec1}

 Let $q=p^r$ be a power of a prime $p$, $\F_q$ be the finite field with $q$ elements and $\F_q^n$ denote the cartesian product of $n$ copies of $\F_q$, for any integer $n\geq 1$.  The ring of polynomials in $n$ variables over $\F_q$ will be denoted by $\F_q[x_1,\ldots,x_n]$. 
 
 \
 
 We say that a  polynomial $f \in \F_q[x_1,\ldots, x_n]$ is a {\it Permutation Polynomial}  (or PP) if the equation $f(x_1,\ldots, x_n)=a$ has $q^{n-1}$ solutions in $\F_q^n$ for each $a\in \F_q$, and  it is  called a {\it Local Permutation Polynomial} (or LPP) if  for each $i$, $1\leq i \leq n$, the  polynomial  $f(a_1,\ldots, a_{i-1},x_i,a_{i+1}, \ldots, a_ n)$ is a permutation polynomial in $\F_q[x_i]$, for all choices of  $(a_1,\ldots, a_{i-1},a_{i+1},\ldots, a_{n}) \in \F_q^{n-1}$. 
 
 \

  It is well known that any map from $\F_q^n$ to $\F_q$ can be uniquely represented as $f \in \F_q[x_1,\ldots,x_n]$ such that $\deg_{x_i}(f) < q$ for all $i=1,\ldots,n$, where $\deg_{x_i}(f) $ is the degree of $f$ as a polynomial in the variable $x_i$ with coefficients in the polynomial ring $\F_q[x_1,\dots,x_{i-1},x_{i+1},\dots, x_n]$, see \cite{LN}. Throughout this paper, we identify all functions $\F_q^n \to \F_q$ with  such polynomials, and every polynomial, will be of degree  $\deg_{x_i}(f) < q$, unless otherwise specified. As a consequence, any polynomial  $f(x_1,\ldots,x_n)$ has degree at most $n(q-1)$.

\

 A classification of permutation polynomials in $\F_q[x_1,\ldots,x_n]$
of degree at most two is given in \cite{N1}, see also \cite{LN, N2} for several properties and results and the particular case $n=1$.   On the other hand, any  LPP is a permutation polynomial, but the opposite is not true in general. We can see that by simply considering the permutation polynomial $f(x_1,\ldots,x_n)=x_1^{q-1}+x_2$, which is not an LPP since $f(x_1,a_2,\dots, a_n)$ takes only the two values 
$a_2$ and $a_2+1$. The author of \cite{M1} and \cite{M2} gives necessary and sufficient conditions for polynomials in two and three variables to be local permutations polynomials over a prime field $\F_p$. These conditions are expressed in terms of the coefficients of the polynomial. 
 
 \
 
We know that any  LPP has degree at most $n(q-2)$, see Proposition 1 in \cite{GU}, and the  authors in \cite{DHK} proved that this bound is sharp for $n=2$ variables, see also \cite{GU} for a short proof of this fact.  A recent result about degree bounds for $n$  local permutation polynomials defining a permutation of $\F_q^n$  is presented in \cite{AKT}.

\

The main contribution in this paper is  a general construction  of permutation and local permutation polynomials of maximum degree  for all $n$.

\

The remainder of the paper is structured as follows. We start with a very short outline of some basic facts about
permutation and local permutation polynomials, and  technical results for later use  in 
 Section~\ref{prelimi}. 
In Section~\ref{pp} we show a class of permutation polynomials  of maximum degree $n(q-1)-1$. 
In Section~\ref{todasvariables} we  provide general constructions  of local permutation polynomials of maximum degree.
In particular we show  the existence of LPP of maximum degree $n(q-2)$ for any $n$, prime number and $q=p^r>3$. 
 Section~\ref{tresvariables} is dedicated to show a different  strategy in order to find  LPP of maximum degree and we  illustrate  it for polynomials with only three variables  providing new families of   local permutation polynomials of maximum degree $3(q-2)$ for  $q=p^r$ and $p=2$ and $p=3$. 
 We conclude with Section~\ref{conclusiones}, which makes some final 
comments and poses open questions.

\section{Preliminaries}\label{prelimi}
One of the main tool for the rest of the paper is the following:
  
\begin{theorem}[\cite{GU}] \label{composition}  Let $f \in \F_q[x_1,\ldots, x_n]$ be a non constant polynomial.

\begin{enumerate}

\item  If   $f = g(x_1,\ldots,x_m)  +h(x_{m+1},\ldots,x_n), \quad  1\leq m < n $, then
\newline
 $f$  is  LPP $\iff$  $g$ and $  h$ are local permutation polynomials.

\item Let   $g(z) \in \F_q[z] $ be  permutation polynomial. Then  
\newline
 $ f $  is a (local) permutation polynomial  $\iff$  $g(f(x_1,\ldots, x_n)$ is a (local) permutation polynomial

\item  Let $h_1(x_1), \ldots, h_n(x_n)$ be permutation polynomials.  Then 
\newline

$ f$  is  (local) permutation polynomial  $\iff$  $f(h_1(x_1),\ldots,h_n(x_n))$ is (local) permutation polynomial

\item The univariate permutation polynomial 
$$t(x)=x+\sum_{k=0}^{q-2}x^k \in \F_q[x]$$ permutes $1$ and $0$, and leave fixed any other element in $\F_q$.

 \item  If $f$ is a  LPP then   $f$  is linear if $q=2$ and $q=3$, and has degree at most  $n(q-2)$ otherwise. 

\end{enumerate}
\end{theorem}
 
  The following Lemmata will be needed later.

\begin{lemma}\label{polinomio1} We consider the univariate polynomial $h \in \F_q[x]$ defined by 
$$h(x)=\sum_{k=0}^{q-3}(k+1)x^k. $$
\begin{enumerate}
\item   $h(x^{q-2}) \equiv \bar {h}(x)  \bmod x^q -x$, where
$$\bar{h}(x)= 1 +x + \sum_{k=0}^{q-2}-kx^k $$
\item Let $t(x)$ be as in  Theorem \ref{composition}, then $\bar{h}(t(x)) \equiv   \bar{h}(x)\bmod x^q -x$.
\end{enumerate}
\end{lemma}
\begin{proof}
The first item follows because  the degree of $x^{k(q-2)}$ is $q-k-1$ for $k=1,\ldots, q-2$. Moreover, for $ a \in \F_q$ with 
 $ a \not=0$ and $ a \not = 1$ we have $\bar{h}(t(a))=\bar{h}(a)$, and  also $\bar{h}(1) = \frac{(q-1)(q-2)}{2}=1=\bar{h}(0).$ In other words $\bar{h}(t(x))$ and $\bar{h}(x)$ represent the same polynomial map.
\end{proof}

\section{Permutation polynomials of maximum degree}\label{pp} 
 We know that any  polynomial $f \in \F_q[x_1,\ldots,x_n]$ has degree at most $n(q-1)$, however for permutation polynomials this bound is smaller.
 
 \
 
\begin{proposition} Any permutation polynomial $f \in \F_q[x_1,\ldots,x_n]$  has degree at most  $n(q-1)-1$.\end{proposition}
\begin{proof}
By the Lagrange  interpolation (see \cite{LN}), we have:
\begin{equation}\label{eq:lag}
f(x_1,\ldots, x_n) = \sum_{(c_1,\ldots,c_n)\in \F_q^n} f(c_1,\ldots,c_n)(1-(x_1-c_1)^{q-1})\cdots(1-(x_n-c_n)^{q-1}),
\end{equation}
The coefficient of the monomial $x_1^{q-1}\cdots x_n^{q-1}$ in the above polynomial identity is
$$ (-1)^n \sum_{(c_1,\ldots,c_n)\in \F_q^n} f(c_1,\ldots,c_n) =  (-1)^n  \sum_{c \in \F_q} q^{n-1}c= 0.$$
The last equality is because $f$ is a permutational polynomial, then for any $a \in \F_q$ the cardinality of  the set
 $$C_a=\{(c_1,\ldots,c_n) \in \F_q^n: f(c_1,\ldots,c_n)=a\}$$
  is $q^{n-1}$.  Moreover the set  $\Delta$ where
 $$\Delta= \{C_a: a \in \F_q\}$$ 
 forms a partition of $\F_q^n$.
\end{proof}

A natural question is if the upper bound $n(q-1)-1$ is sharp.  We are now ready to prove the main result in this section.
\begin{theorem} For any $n$, prime number $p$ and $q=p^r$, there is an PP over $\F_q[x_1,\dots, x_n]$ of maximum degree $n(q-1)-1$
 defined over $\F_p$.
\end{theorem}

\begin{proof}
We consider  $t(x)=x+\sum_{k=0}^{q-2}x^k \in \F_q[x]$ (see Theorem \ref{composition}) and constructing the polynomial
$$ h_n=h_n(x_1,\ldots,x_n)= x_1^{q-1} \cdots x_{n-1}^{q-1} (t(x_n)-x_n) +x_n$$
Let $a\in \F_q$, for any $(c_1,\ldots,c_{n-1}) \in \F_q^{n-1}$ we prove that there exits a unique $c_n\in \F_q$ such that 
$$h(c_1,\ldots,c_n) = a$$
Two cases are presented:
\begin{enumerate}
\item If $c_i \not=0$ for all $i=1,...,n-1$, we have $h_n(c_1,\ldots,c_{n-1},x_n)= t(x_n)$ which is a univariate permutation polynomial.
\item Otherwise, if $c_i=0$ for some $i$, then $h_n(c_1,\ldots,c_{n-1},x_n)= x_n$, which is a univariate permutation polynomial, since it is linear.
 \end{enumerate}
 \end{proof}

\begin{remark}
Of course, there are  other families of permutation polynomials  of maximum degree, for instance: 
\begin{itemize}
\item Let $q=p^r$ with $p>2$ and considering the polynomial

$$ x_1^{q-1} \cdots x_{n-1}^{q-1} x_n^{q-2} +x_n^{q-2}$$
 
\item In  even extensions of $\F_2$, $\F_{2^{2r}}>4$,  we take the Dickson polynomial  ($ 0 \not= a \in \F_q$)
$$g_k(x,a) = \sum_{j=0}^{\left \lfloor{\frac{k}{2}}\right \rfloor} \frac{k}{k-j}\binom{k-j}{k}(-a)^jx^{k-2j},$$
which is a univariate permutation polynomial in $\F_q$ if and only if $\gcd(k, q^2-1) =1$, see \cite{LN}.

 Then, the polynomial $g_{q-2}(x,1)$  is a permutation polynomial of degree $q-2$ and it is of the form:
 
 $$ g_{q-2}(x,1) = x^{q-2}+m(x) + x^2,$$  where  the degree of $m(x)$ is bigger than $2$  and smaller than $q-2$ or $m(x)=0$.
 
 The polynomial
 $$ x_1^{q-1} \cdots x_{n-1}^{q-1} (x_n^{q-2}+m(x_n)  ) + x_n^2$$
 is a permutation polynomial since $\gcd(2^{2r}-2, 2^{4r}-1)=1$, and has degree $n(q-1)-1$.
 
 \
 
For $q=4$, we consider the polynomial $\alpha_n= \alpha_n(x_1,\ldots,x_n) \in \F_4[x_1,\ldots,x_n]$ which is the sum of all  monomials of degree smaller than $n(q-1)$ plus $x_1$. For $n=2$, we have
 $$\alpha_2 = x_{2}^{3} x_{1}^{2} + x_{2}^{2} x_{1}^{3} + x_{2}^{3} x_{1} + x_{2}^{2}
x_{1}^{2} + x_{2} x_{1}^{3} + x_{2}^{3} + x_{2}^{2} x_{1} + x_{2}
x_{1}^{2} + x_{1}^{3} + x_{2}^{2} + x_{2} x_{1} + x_{1}^{2}  +x_2 +1
$$
And  more general, for arbitrary $n \geq 1$: 
 
 $$\alpha_n = \alpha_n(x_1,\ldots,x_n)=  \sum_{0\leq i_1\leq 3,\ldots,  0\leq i_n\leq 3}^{i_1+\cdots +i_n \leq n(q-1)-1} x_1^{i_1}x_2^{i_2}\cdots x_n^{i_n}  +x_1$$
 
 In order to prove that  $\alpha_n(x_1,\ldots,x_n) $  is a PP in $\F_4[x_1,\ldots,x_n]$ of maximum degree $3n-1$ for $n \geq  1$, we observe that 
 
   \begin{eqnarray*} \label{relation}
 \alpha_n(x_1,\ldots,x_{n-1},1)&=&x_1^3\cdots x_{n-1}^3 +x_1 \\
 \alpha_n(x_1,\ldots,x_{n-1},0)&=&x_1^3\cdots x_{n-1}^3+ \alpha_{n-1}(x_1,\ldots,x_{n-1}) \\
  \alpha_n(x_1,\ldots,x_{n-1},u)&=&\alpha_{n-1}(x_1,\ldots,x_{n-1}) \\
   \alpha_n(x_1,\ldots,x_{n-1},u+1)&=&\alpha_{n-1}(x_1,\ldots,x_{n-1}) 
\end{eqnarray*}
where $\F_4 = \{ 0, 1, u, u+1 \}$ such that $u^2+u+1 = 0$. Now, it is straightforward  to prove  it by  induction  on $n$. 

\item Let $q=p^r$ odd, and consider $g\in F_q[x_1,\ldots, x_n]$ any polynomial of total degree $n(q-1)/2$, $f\in\F_q[y]$ a PP of degree $q-2 $ and $a\in \F_q$ a quadratic non-residue. Observe that
$0=x^{q-1}-1=(x^{(q-1)/2}-1)(x^{(q-1)/2}+1)$ has exactly $q-1$ solutions over $\F_q$ and both  $x^{(q-1)/2}-1=0$ and $x^{(q-1)/2}+1=0$ have at most $(q-1)/2$ solutions each, so in particular $x^{(q-1)/2}-1=0$ and $x^{(q-1)/2}+1=0$
have both exactly $(q-1)/2$ solutions, and hence there are $(q-1)/2$  quadratic non residues. Now consider the polynomial $h\in\F_q[x_1,\ldots, x_n,y]$ given by
$$
h(x_1,\dots, x_n,y)=(g(x_1,\dots,x_n)^2-a)f(y),
$$
where $a$ is not a square. Then, for any $(c_1,\cdots c_n)\in\F_q^n$ and $c\in\F_q$ there is exactly one $b\in\F_q$ such that 
$$
h(c_1,\dots,c_n,b)=c.
$$
Indeed, since $a$ is not a square $g(c_1,\dots,c_n)^2-a=\delta\ne 0$ and given $c\in\F_q$ there is a unique $b$ such that $f(b)=c\delta^{-1}$, since $f$ is PP.
In particular $h(c_1,\dots,c_n,b)=c$ and $h(x_1,\dots,x_n,y)=c$ has exactly $q^{n}$ solutions as needed.

\

Now suppose $q=2^r$ with $r$ even. Then, $3|q-1$ and arguing as before, we find $\frac{2(q-1)}3$ elements in $\F_q$ which are not cubes. In fact all the solutions of $r(x)=0$ where 
$x^{q-1}-1=(x^{(q-1)/3}-1)r(x)$. So we can select  $a\in\F_q$ which is not a cube and construct
$$
h(x_1,\dots,x_n,y)=(g(x_1,\dots,x_n)^3-a)f(y),
$$
where $g\in F_q[x_1,\ldots, x_n]$ is  any polynomial of total degree $n(q-1)/3$, and $f\in\F_q[y]$ a PP of degree $q-2$. The proof is analogous to the previous case.

\

Finally, suppose $q=2^r$ with odd $r>1$. Observe that, in this case, we could have $q-1$ a Mersenne prime and the previous argument does not work since $x^r$ is permutation of $\F_q$ for any  $r<q-1$.  In this case consider
$$
h(x_1,\dots, x_n,y)=(x_1^{q-1}\cdots x_n^{q-1}+\alpha)f(y),
$$
where $\alpha\ne 0,1$ and $f$ is any PP of degree $q-2$. Then for any $(c_1,\dots, c_n)\in\F_q^n$ we have $(x_1^{q-1}\cdots x_n^{q-1}+\alpha=\delta\ne 0$ and hence for each $c\in\F_q$ there exists a unique $b$ tal que $f(b)=c\delta^{-1}$. The proof now  that $h$ is a PP of degree $n(q-1)+q-2$ is similar to those above.

 \end{itemize}
 
  \end{remark}
 
\section{Local permutation polynomials of maximum degree in $\F_q[x_1,\ldots,x_n]$}\label{todasvariables}

This section is dedicated to show  local permutation polynomials of maximum degree in $n$ variables over a finite field $\F_q$. 
As many times in algebra we treat differently the characteristic 2.

\begin{theorem}\label{lpp2} Let $q=2^r > 2$.  The polynomials

$$\beta_{q}^{(n)}=\beta_{q}^{(n)}(x_1,\ldots,x_n)=  \sum_{i_1\geq 1,\ldots,  i_n\geq 1}^{2^r-2} x_1^{i_1}x_2^{i_2}\cdots x_n^{i_n}  + x_1+\cdots+ x_n$$
are LPP of $\F_q[x_1,\ldots,x_n]$ of degree $n(2^r-2)$ for $n \geq  1$.
\end{theorem}
\begin{proof}
Since the degree is in fact $n(q-2)$, we just need to prove that it is indeed an LPP, but this follows from the fact that 
\begin{itemize}
\item  $\beta_{q}^{(n)}(x_1, \ldots, x_{n-1}, a) = x_1+\cdots+ x_{n-1} +a$, for $a =0$ and $a = 1$.
\item  $\beta_{q}^{(n)}(x_1, \ldots, x_{n-1}, u^i) = \beta_{q}^{(n-1)}(x_1, \ldots, x_{n-1}) +u^i$, for $i=1,\ldots, q-2$
\end{itemize}
where  $u$ is a primitive element of $\F_q$, i.e.  $\F_q= \{ 0, 1, u, u^2,\ldots, u^{q-2} \}$, and induction, since  the polynomial $\beta_q^{(2)} \in \F_q[x_1,x_2]$ is LPP.

\end{proof}

To prove our result for odd characteristic, we first state a simple lemma which allows to select properly the number of variables.

\medskip


\begin{lemma}\label{lem:less} If there is an LPP polynomial $f(x_1\dots,x_n)\in\F_q[x_1,\dots,x_n]$ of maximum degree  then there is an LPP polynomial of maximum degree for any $m\le n$.
\end{lemma}
\begin{proof} Let the polynomial $f(x_1\dots,x_n)=ax_1^{q-2}\cdots x_n^{q-2}+P$ for some $P$ of smaller degree. We know that for any $\alpha \in \F_q$ we have $f(\alpha,x_2,\dots,x_n)=a\alpha^{q-2}x_2^{q-2}\cdots x_n^{q-2}+P$ is an LPP polynomial over $\F_q$. If it has not degree $(n-1)(q-2)$ is because $P$ has the monomial $-a\alpha^{q-2}x^{q-2}\cdots x_n^{q-2}$, but then  $f(\beta,x_2\dots,x_n)$ has the monomial  $(a\beta^{q-2}-a\alpha^{q-2})x_2^{q-2}\cdots x_n^{q-2}\ne 0.$
\end{proof}
We are ready to show the main result os this section.

\medskip

\begin{theorem}\label{teo:divis} Let $q$ be such that $(b,q-1)=1$ for some $1<b<p-1$. Then for any $n\ge 1$ there is an LPP over $\F_q[x_1,\dots, x_n]$ of maximum degree, defined over $\F_p$. 
\end{theorem}
\begin{proof} Consider the polynomial 
$$
f_1(y_1,\dots,y_b)=(y_1+\dots+y_b)^b,
$$
which is an LPP by Theorem \ref{composition}. Since it is a form,  we have
$$
f_1:=f_1(y_1,\dots,y_b)=\sum_{a_1+\dots+a_b=b}y_1^{a_1}\cdots y_{b}^{a_b}=b!y_1\cdots y_b+r(y_1,\dots, y_b),
$$
where 
$$
r(y_1,\dots, y_b)=\sum_{\substack{a_1+\cdots+a_b=b \\ a_1\cdots a_b=0}}y_1^{a_1}\cdots y_{b}^{a_b}.
$$
Note that for $b$ positive integers to add up to b, either all of them are $1$ or at least one must be cero. In particular every monomial of $r(y_1,\dots, y_b)$ has less than $b$ variables. Now consider the sequence of polynomials 
$$
f_{i+1}=f_{i+1}(y_1,\dots, y_{b^{i+1}})=\left(f_{i,0}+\dots+f_{i,b-1}\right)^b,
$$
where
$$
f_{i,k}=f_i(y_{kb^{i}+1},\dots, y_{(k+1)b^i}),
$$
for $k=0,\dots b-1$.
By induction, we see that 
$$
f_i=b!^{\frac{b^i-1}{b-1}}y_1\cdots y_{b^i}+r_i(y_1,\dots, y_{b^i}),
$$
where all the monomials of $r_i(y_1,\dots, y_{b^i})$ have less than $b^i$ variables.
We have already done the case for $i=1$. Now, 
$$
f _{i+1}=b!f_{i,0}\cdots f_{i,b-1}+r(f_{i,0}\cdots f_{i,b-1})
$$
where on each monomial of $r(f_{i,0}\cdots f_{i,b-1})$ there is one of the $f_{i,k}$ missing and hence, has less than $b^{i+1}$ variables. We apply now induction to get 
$$
f_{i+1}=b!^{\frac{b^{i+1}-1}{b-1}}y_1\cdots y_{b^{i+1}}+r_{i+1}(y_1,\dots, y_{b^{i+1}}).
$$
where  each monomial of $r_{i+1}(f_{i,0}\cdots f_{i,b-1})$  has less than $b^{i+1}$ variables.
Now, consider the ideal on $\F_q[x_1,\dots,x_{b^k}]$ given by $\mathcal I=<x_1^{q}-x_1,\dots, x_{b^k}^{q}-x_{b^k}>$ and  consider the polynomial
$$
F(x_1,\dots, x_{b^k})=f_k(x_1^{q-2},\dots, x_{b_k}^{q-2}) {\pmod {\mathcal  I}}
$$
Since the only element of degree $n(q-2)$ of a polynomial with $n$ variables reduced modulo $\mathcal I$, must have the monomial $x_1^{q-2}\cdots x_n^{q-2}$, we see that deg$F(x_1,\dots, x_{b^k})$ is indeed $b^k(q-2)$ since, this monomial is not in $r_k$ and $b!^{\frac{b^{i+1}-1}{b-1}}$ is non zero modulo $p$. Once we have the theorem for polynomials with $b^k$ variables, apply Lemma \ref{lem:less} to get the result for any $n$.
\end{proof}

 Observe that the theorem do not cover every $q$. Just consider $q=p^r$ where $\varphi((p-2)!)|r$. Then, every $1<b<p-1$ is a divisor of $q-1$ by Euler's Theorem, since $(q,(p-2)!)=1$. We will attack the remaining cases in the following way. 
 
 \
 
 \begin{theorem}\label{lppp} Let $q=p^r > 3$ for $p\ge 3$ prime and  $r$ a natural number. There exist an LPP $f(x_1,\ldots,x_n) \in  \F_q[x_1,\ldots,x_n]$ of maximum degree $n(q-2)$. 
\end{theorem}

\begin{proof} The result is contained in Theorem \ref{teo:divis}  for $\F_p[x_1,\ldots, x_n]$ so we can assume that $q=p^r>3$ and $r\ge 2$.  Aplying the Lagrange interpolation formula  (\ref{eq:lag}), we write  the polynomial $f(x)\in\F_q[x]$  as  the linear combination
\begin{equation*}\label{eq:polin}
f(x)=\sum_{i=0}^{q-1}\alpha_ig_i(x),
\end{equation*}
where $\{a_0,\dots, a_{q-1}\}=\F_q$,  $f(a_i)=\alpha_i$, and $g_i=1-(x-a_i)^{q-1}$, for $i=0,\dots, q-1$. The following lemma follows directly from the definition of $g_0,\dots, g_{q-1}$.

\medskip

\begin{lemma}\label{lem:deg} With the notation as above, the polynomial  $f(x)$ has degree $q-2$ if and only if $\sum_{i=0}^{q-1}{\alpha_i}=0$,and  $\sum_{i=0}^{q-1}a_i\alpha_i\ne 0$.
\end{lemma}

\
Now, we consider, with the notation as above,  $a_i=i$, for $i=0,\dots, p-2$ and  $a_{p-1}=\beta\notin\F_p$, and the values $\alpha_i=1$ for $i=0,\dots,p-1$ and $\alpha_j=0$ otherwise. Then  $\sum_{j=0}^{q-1}\alpha_i=0$, while 
$$
\sum_{j=0}^{q-1} a_i\alpha_i=\beta+\sum_{j=0}^{p-2} i=\beta+\frac{(p-1)(p-2)}{2}=\beta+1+p\frac{p-3}{2}=\beta+1\ne 0,
$$ 
so the polynomial $p(x)=\sum_{i=0}^{p-1}g_i(x)$ has degree $q-2$ by Lemma \ref{lem:deg} and  $p(x)=0$ if $x\notin Z$, where $Z=\{a_0,\dots, a_{p-1}\}$. Consider now  the polynomial
$$
f(x_1,\dots,x_n)=\prod_{i=1}^np(x_i)+\sum_{i=1,\dots, n}t_{\beta}(x_i),
$$
where $t_{\beta}(x)$ is the transposition that sends $t_\beta(\beta)=p-1$ and $t_\beta(p-1)=\beta$.  The degree of $f$  is $n(q-2)$. Moreover, since it is symmetric we just need to prove that the polynomial  $F(x)=f(x,c_2,\dots,c_n)$ is a permutation polynomial for any selection of $c_2\dots,c_n$. Now, if $\{c_2,\dots, c_n\}\not\subset Z$  then $F(x)=t_\beta(x)+C$, which is a permutation polynomial. On the other hand, if $\{c_2,\dots c_n\}\subset Z$, then
$F(x)=p(x)+t_\beta(x)+C$. If $\{x,y\}\subset \F_q\setminus Z$,  or  $\{x,y\}\subset Z$, then clearly $F(x)\ne F(y)$, so it remains to prove that if $x\in Z$ and $y\notin Z$, then $F(x)\ne F(y)$. Now, for $i=0,\dots, p-2$, the value of $F$ is given by $F(a_i)=1+t_{\beta}(a_i)+C=i+1+C$, while  $F(y )=t_\beta(y)+C$, so if they are equal $t_{\beta}(y)=i+1$, so it must be that $y\in Z$. Finally $F(\beta)=C=t_{\beta}(y)+C$ only if $y=0\in Z$, so $F(x)$ is indeed PP.

\end{proof}


\

\section{Local permutation polynomials of maximum degree in $\F_q[x_1,x_2,x_3]$}\label{tresvariables}

In this section we show a different  strategy in order to find  LPP of maximum degree and we illustrate it  for polynomials with only three variables  providing new families of   LPP of maximum degree $3(q-2)$ for  $q=p^r>3$ and $p=2$ and $p=3$. 

\

For $i = 1,\ldots, n$  we consider the following recurrence polynomial relation:
\begin{equation} 
\begin{split}
\label{relation}
&f_1 = x_1 \\
&f_i = f_i(x_1,\ldots,x_i) = t(f_{i-1}(x_1,\ldots, x_{i-1})^{q-2} + x_i^{q-2})=t(f_{i-1}^{q-2}+x_i^{q-2}),
\end{split}
\end{equation}
where $t(x)$ is as in Theorem \ref{composition}.  Since $\gcd(q-1, q-2)=1$ then  $x^{q-2}$  is a permutation polynomial, see \cite{LN}.  By Theorem \ref{composition} we have $f_i \in \F_p[x_1,\ldots,x_n]$  is an LPP in  $\F_q[x_1,\ldots,x_n]$  for $i=1,\ldots, n$. 

\

 We also  know:

\begin{theorem}[\cite{GU}] \label{jorge} For any $q=p^r >3$ and  $p \not=2$ the polynomial $f_2(x_1,x_2)$ has degree $2(q-2)$.
\end{theorem}

This section shows that $f_3(x_1,x_2,x_3)$ has degree $3(q-2)$ when $q=p^r>3$ and $p \not=2$.

\begin{theorem}\label{threecase} With the above notations and definitions we have  $f_3(x_1,x_2,x_3)$ has degree $3(q-2)$ when $q=p^r>3$ and $p \not=2$.
\end{theorem}
\begin{proof}
$$
f_3=t(f_2^{q-2}+x_3^{q-2})= f_2^{q-2} +x_3^{q-2} +\sum_{k=0}^{q-2}(f_2^{q-2}+x_3^{q-2})^k.
$$

\begin{eqnarray*}
f_3&=&f_2^{q-2}+x_3^{q-2} + \sum_{k=0}^{q-2}\sum_{j=0}^{k}\binom{k}{ j}f_2^{(q-2)(k-j)}x_3^{j(q-2)}\\
&=&f_2^{q-2} +y^{q-2}+\sum_{j=0}^{q-2}\left(\sum_{k=j}^{q-2}\binom{k}{ j}f_2^{(q-2)(k-j)}\right)x_3^{j(q-2)}.
\end{eqnarray*}
Now we have $j(q-2)\equiv q-2\pmod{q-1}$ only if $j\equiv 1\pmod {q-1}$. Then the coefficient $C$ in $f_3$ of  $x_3^{q-2}$  is 
$$C=\sum_{k=1}^{q-2}\binom{k}{ 1}f_2^{(q-2)(k-1)} = \sum_{k=0}^{q-3}(k+1)f_2^{(q-2)k} = h(x^{q-2})(f_2)=$$
by Lemma \ref{polinomio1}, we obtain that

$$C=\bar{h}(f_2), \quad {\text where} \quad   \bar{h}(x)= 1 +x + \sum_{k=0}^{q-2}(q-k)x^k$$

Again, by  item (2) of Lemma \ref{polinomio1} we get:
 
$$C= \bar{h}(f_2) = \bar{h}(t)(x_1^{q-2}+x_2^{q-2})=\bar{h}(x_1^{q-2}+x_2^{q-2})=$$
$$= 1+x_1^{q-2}+x_2^{q-2} + \sum_{k=0}^{q-2}(q-k)(x_1^{q-2}+x_2^{q-2})^k$$

\begin{eqnarray*}
C&=&1+x_1^{q-2}+x_2^{q-2}+\sum_{k=0}^{q-2}(q-k)\sum_{j=0}^{k}\binom{k}{ j}x_1^{(k-j)(q-2)}x_2^{j(q-2)}\\
&=&1+x_1^{q-2}+x_2^{q-2}+\sum_{j=0}^{q-2}\left(\sum_{k=j}^{q-2}(q-k)\binom{k}{ j}x_1^{(k-j)(q-2)}\right)x_2^{j(q-2)}.
\end{eqnarray*}
Now we have $j(q-2)\equiv q-2\pmod{q-1}$ only if $j\equiv 1\pmod {q-1}$. Selecting $k=2$ we have that $C$ has the term
$$ -4x_1^{q-2}x_2^{q-2}\ne 0.
$$
For any other $j\ne 1$, $j(q-2)\not\equiv q-2 \pmod{q-1} $, meanwhile for $j=1$ and any other $k$ we have that $(k-j)(q-2)\not\equiv q-2\pmod {q-1}$  and hence  $-4x_1^{q-2}x_2^{q-2}$ is the only monomial of degree $2(q-2)$.

\
As consequence  $-4x_1^{q-2}x_2^{q-2}x_3^{q-2}$ is  the only monomial of degree $3(q-2)$ of the polynomial $f_3$ and this conclude the proof.

\end{proof}

Using a similar proof strategy we obtain that $f_4(x_1,x_2,x_3,x_4)$ has degree $4(q-2)$ and also the  following:

\

\begin{theorem} With the above notations and definitions we have:
\begin{enumerate}
\item The polynomial $t(f_2+x_3^{q-2}) \in \F_q[x_1,x_2,x_3] $ has maximum degree $3(q-2)$ if $q=p^r$ and $p \not=2$ and $p \not=3$.
\item The polynomial $t(\bar{f}_2+x_3^{q-2}) \in \F_q[x_1,x_2,x_3] $  has  maximum degree $3(q-2)$ when $q=3^r > 3$,  where  $\bar{f}_2=\bar{f}_2(x_1,x_2) = t(x_1^{q-4}+x_2^{q-2})$.
\item   Let $q=2^r > 2$ and $s= (q-2)/2$. The polynomial $f= f(x_1,x_2,x_3)= (h_2 + x_3^{s})^{q-2} \in \F_q[x_1,x_2,x_3]$ is LPP of maximun degree $3(q-2)$ where $h_2 = h_2(x_1, x_2) = t(x_1^{q-2} + x_2^{s})$. 
\end{enumerate}
\end{theorem}
\begin{proof}
To prove the last statement   let $q=2^r$ for $r>2$.

$$f =(f_2 + x_3^{s})^{q-2}= \sum_{j=0}^{q-2} \binom{q-2}{ j}f_2^{q-2-j}x_3^{js} $$

Now  $js \equiv q-2 \pmod{q-1}$ if only if $j=2$. Then, the coefficient $C$ of $x_3^{q-2}$ in the polynomial $f$ is

$$ C= \binom{q-2}{2}f_2^{q-4}= f_2^{q-4} = t^{q-4}(x_1^{q-2}+x_2^{s})$$

Because  $t^{g-4} \equiv  t +x^{q-4} +x \mod x^q -x $, then

$$C = (x_1^{q-2}+x_2^{s})^{q-4}+x_1^{q-2}+x_2^{s}+ t(x_1^{q-2}+x_2^{s})$$

So, we are looking  for the  coefficient   of the monomial $x_2^{q-2}$ in the polynomial  $C$.
 Clearly  it  only appears  in the  polynomial  $t(x_1^{q-2}+x_2^{s})$, because
 $$ (x_1^{q-2}+x_2^{s})^{q-4} = \sum_{j=0}^{q-4}\binom{q-4}{ j}x_1^{(q-4-j)(q-2)}x_2^{js},
$$
 but $js\equiv q-2 \pmod{q-1}$ only if $j=2$ and $\binom{q-4}{ j}=\binom{q-4}{ 2}=0$.
 
Now,
 $$ t(x_1^{q-2}+x_2^{s}) =x_1^{q-2}+x_2^{s}+\sum_{k=0}^{q-2}(x_1^{q-2}+x_2^{s})^k=$$
 $$= x_1^{q-2}+x_2^{s}+\sum_{j=0}^{q-2}\left(\sum_{k=j}^{q-2}\binom{k}{ j}x_1^{(k-j)(q-2)}\right)x_2^{js}.
$$
Again $js\equiv q-2 \pmod{q-1}$ only if $j=2$ and on the other hand we have that $(k-2)(q-2)\equiv q-2 \pmod{q-1}$ only if $k=3$, so the only term in $ C$ of degree $2(q-2)$ is $x_1^{q-2}x_2^{q-2}\ne 0$, and with this the proof is concluded.
\end{proof}

\

\section{Conclusions and open problems}\label{conclusiones}
Contrary to the many papers and results on permutation polynomials in one variable, there are few  for permutation and local permutation polynomials in several variables. 

\medskip

We  have presented several constructions of  permutation and local permutation polynomials $f\in \F_q[x_1,\ldots, x_n]$, in $n$ variables, of maximum degree generalising  previous results of the two variables case.  Giving others families of LPP of maximum degree and providing applications to latin hypercubes is an interesting problem, see  \cite{DHK} for the relation between Latin Squares and  LPP of maximum degree $2(q-2)$. 

\

We believe that Theorem \ref{threecase} can be generalized:

\begin{conjecture} With the above notations and definitions we have  $f_n(x_1,\ldots,x_n)$ has degree $n(q-2)$ when $q=p^r>3$ and $p \not=2$.
\end{conjecture}

\

Observe  the above  polynomial $f_n(x_1,\ldots,x_n)$ is defined over the prime finite field $\F_p$ and we have constructed LPP  in $\F_q$ of maximum degree defined over the prime finite field $\F_p$ for some $q=p^r>3$, see for instance Theorem \ref{lpp2}, Theorem \ref{teo:divis},  but not for all, see Theorem \ref{lppp}.   Giving   families of LPP  in $\F_{q}[x_1,\ldots,x_n]$  for any $q$ of maximum degree defined in  $\F_{p}[x_1,\ldots,x_n]$ is a challenging open  problem as well.

\end{document}